\documentclass[10pt]{article}

\usepackage{amssymb}
\usepackage{amsmath}
\usepackage{amsthm}
\usepackage{graphicx}
\usepackage{siunitx}
\usepackage{mathtools}
\usepackage{tikz}
\usetikzlibrary{arrows,calc}
\usetikzlibrary{shapes,arrows,shadows,arrows.meta,decorations.pathreplacing,angles,quotes}
\usepackage{bm}
\usepackage{booktabs}
\usepackage{pgfplots}
\usepackage{float}

\newtheorem{theorem}{Theorem}[section]
\newtheorem{lemma}{Lemma}[section]
\newtheorem{example}{Example}[section]

\newcommand{\R}{\mathbb{R}}
\newcommand{\N}{\mathbb{N}}
\newcommand{\cU}{\mathcal{U}}
\DeclareMathOperator*{\argmin}{arg\,min}
\newcommand{\eins}{\mbox{1\hskip-0.24em l}}
\def\T{^{\sf T}}
\newcommand{\Oh}{\mathcal{O}}

\author{D.~Frenzel, J.~Lang}
\title{A Third-Order Weighted Essentially Non-Oscillatory Scheme in Optimal Control
Problems Governed by Nonlinear Hyperbolic Conservation Laws}
\author{
David Frenzel\\
{\small \it Technical University of Darmstadt, Department of Mathematics,} \\
{\small Dolivostra{\ss}e 15, 64293 Darmstadt, Germany}\\
{\small frenzel@gsc.tu-darmstadt.de} \\ \\
Jens Lang\footnote{corresponding author, ORCID: 0000-0003-4603-6554}\\
{\small \it Technical University of Darmstadt, Department of Mathematics,} \\
{\small Dolivostra{\ss}e 15, 64293 Darmstadt, Germany}\\
{\small lang@mathematik.tu-darmstadt.de}}
\begin{document}
\maketitle

\begin{abstract}
The weighted essentially non-oscillatory (WENO) methods are popular and effective spatial discretization methods for nonlinear hyperbolic partial differential equations. Although
these methods are formally first-order accurate when a shock is present, they
still have uniform high-order accuracy right up to the shock location. In this paper,
we propose a novel third-order numerical method for solving optimal control problems subject to scalar nonlinear hyperbolic conservation laws. It is based on the first-disretize-then-optimize
approach and combines a discrete adjoint WENO scheme of third order with the classical strong stability preserving three-stage third-order Runge-Kutta method SSPRK3. We analyze its approximation properties and apply it to optimal control problems of tracking-type with non-smooth target states. Comparisons to common first-order methods such as the Lax-Friedrichs and Engquist-Osher method show its great potential to achieve a higher accuracy along with good resolution around discontinuities.\\
\end{abstract}



\section{Introduction}
We consider the optimal control problem
\begin{flalign}\label{problem}
u_0^{min} = \argmin_{u_0 \in \cU_{ad}} J(y(T,\cdot;u_0),y_d)
\end{flalign}
with the tracking-type functional
\begin{flalign}\label{cF}
J(y(T,\cdot;u_0),y_d) = &\;\int_{I} G(y(T,x;u_0),y_d(x))\, dx,
\end{flalign}
where
\begin{flalign}\label{{LSCost}}
G(y(T,x;u_0),y_d(x)) = &\; \frac{1}{2} |y(T,x;u_0)-y_d(x)|^2
\end{flalign}
and
$y=y(T,x;u_0)$ is the scalar entropy solution at the final time $T>0$
of the nonlinear hyperbolic conservation law (later referred to as state equation)
\begin{equation}
\begin{aligned}\label{hypEqIntro}
\partial_ty + \partial_xf(y) = &\; 0, \quad\quad (t,x) \in \Omega_T \coloneqq (0,T] \times \R, \\[1mm]
y(0,x) = &\; u_0(x), \ x \in \R .
\end{aligned}
\end{equation}
Here, $u_0\in\cU_{ad}\subseteq L^{\infty}(\R)$ is the control
and $y_d \in L^2(\R)$ denotes a given target towards which we strive to optimize. We
assume that the flux function satisfies $f \in C^m(\R)$ with sufficiently large $m\in\N$
and is convex, the admissible set $\cU_{ad}$ is non-empty, convex and closed, and
the region of integration $I$ in (\ref{cF}) is a bounded interval.
Weak solutions to (\ref{hypEqIntro}) are in general not unique, which implies that the
physically relevant solution has to be chosen. As a fact we cite the well-known result
from \cite{Kruzkov1970}, which states that for $u_0 \in L^{\infty}(\R) \cap BV(\R)$
there exists a unique entropy solution in the sense of Kr\v{u}zkov in the class
$C([0,T],L_{loc}^1(\R)) \cap L^{\infty}(\R \times [0,T])$.

In this work, we focus on the numerical treatment of optimal control problems (\ref{problem})
governed by hyperbolic conservation laws, which has been studied amongst others in
\cite{AguilarSchmittUlbrichMoos2019,BandaHerty2012,CastroPalaciosZuazua2008,ChertockHertyKurganov2014,
Giles2003,GilesUlbrich2010a,GilesUlbrich2010b,
HajianHintermuellerUlbrich2017,HertyKurganovKurochkin2015,
HertyKurganovKurochkin2018,Kurochkin2015,LecarosZuazua2014,LecarosZuazua2016,
Ulbrich2001,Ulbrich2002,Ulbrich2003}.
We will follow the \textit{first-disretize-then-optimize} approach, i.e., equation (\ref{hypEqIntro})
is first discretized in space and time by applying a weighted essentially non-oscillatory (WENO) scheme and a strong stability
preserving Runge-Kutta (SSPRK) method. This leads to a finite dimensional optimal control
problem, for which the first-order discrete optimality system can be derived and solved
by existing optimization solvers such as nonlinear Newton-type algorithms. In spite of the
large size of the resulting problems, the flexibility of this approach naturally allows the
incorporation of additional constraints and bounds. Further advantages are the direct use
of automatic differentiation techniques and the computation of discrete adjoints, which are
consistent with the discrete optimal control problem. Symmetric approximations of
Hessian matrices can be easily derived and result in a computational speedup.

The application of common methods from nonlinear optimization requires the computation
of directional derivatives of the target functional $J$ with respect to the control.
An efficient computation
of the gradient can be effectuated by using the so-called adjoint approach, in which
the derivative is represented via the adjoint state. The crucial issue of hyperbolic
conservation laws is the possible formation of shocks even for smooth initial data,
for which reason the classical adjoint calculus does not apply. To overcome these
difficulties, nonstandard variational concepts have been developed in
\cite{BressanMarson1995,Ulbrich2001,Ulbrich2003}, which incorporate the shock sensitivity
in order to derive rigorous optimality conditions. The resulting non-conservative equation
has been studied in \cite{BouchutJames1998,Conway1967,Ulbrich2001}. Their numerical
resolution is intricate, since the interior boundary condition defined on a set of
Lebesgue measure zero -- existing for the continuous setting -- is not present for the
discrete counterpart. This inherent problem has been addressed in \cite{AguilarSchmittUlbrichMoos2019,Giles2003,GilesUlbrich2010a,GilesUlbrich2010b}.
The theory is, however, restricted to differentiable monotone schemes which have
sufficiently large numerical diffusion and are of first order only.

To avoid unwanted smearing of the solution by large numerical diffusion and to overcome the lower
order restriction of monotone schemes, we propose a novel approach based on WENO schemes introduced in \cite{JiangShu1996,LiuOsherChan1994,Shu1998}. These schemes
have been proven to approximate hyperbolic equations comprising both shocks and complex smooth solution structure with higher accuracy and adequate stability along with good resolution around discontinuities. Although these methods are formally first-order accurate when a shock is present, they still have uniform high-order accuracy right up to the shock location.
By employing a global flux-splitting, the numerical flux function becomes classically differentiable
and therefore allows to develop discrete adjoint WENO methods of higher order. Since the
third-order WENO method is often applied in applications, we consider this method in the context
of optimal control in more detail. We prove that the discrete adjoint WENO3 method is third-order consistent in space for smooth solutions. A fully discrete method is derived by applying a third-order SSPRK method.
We present numerical results and study the approximation behaviour of the adjoint WENO3 scheme.
Finally, we solve an optimal control problem with discontinuous target and compare the performance
of our novel scheme to common first-order schemes such as the modified Lax-Friedrichs and the Engquist-Osher scheme. Further examples can be found in \cite{Frenzel2020}.

\section{Adjoint Equation and Reversible Solutions}
In this section, we briefly recall some theoretical basics in order to set up appropriate
adjoint equations for hyperbolic conservation laws. As pointed out in \cite[Example\,1]{BressanMarson1995},
the solution operator $S_t: u_0 \mapsto y(t,\cdot;u_0)$ is generically not differentiable in
$L_{loc}^1(\R)$, for which reason the classical adjoint calculus does not apply. However, in \cite{Ulbrich2001} it has been shown that entropy solutions to hyperbolic conservation laws admit a generalized differentiable structure called \textit{shift-differentiability}. Under suitable assumptions, a generalized Taylor expansion in $L_{loc}^1$ of the form
\begin{flalign}
y(t, \cdot ; u_0 +\delta u_0) =y(t,\cdot ; u_0)+ S_y^{(x_i)}(T_{u_0}(\delta u_0))(\cdot)
+o(\| \delta u_0 \|_{L^{\infty}(\R)} )
\end{flalign}
exists for all $\delta u_0 \in L^{\infty}(\R)$,
where $T_u: \delta u \in L^\infty(\R) \mapsto (\delta y^T, \delta x_1,...,\delta x_N) \in L^r(I)\times\R^N$,
$r\in (1,\infty]$, is a bounded linear operator and $S_y^{(x_i)}$ is the shift variation defined by
\begin{flalign}
S_y^{(x_i)}(\delta y^T, \delta x_1,\ldots,\delta x_N)(x) = \delta y^T(x)+ \sum_{i=1}^N (y(T,x_i-)-y(T,x_i+))\,\text{sign}(\delta x_i) \eins_{\Omega_i},
\end{flalign}
where $\Omega_i=[\min(x_i,x_i+\delta x_i),\max(x_i,x_i+\delta x_i)]$ and
$x_1,...,x_N$ denote the locations of the down-jumps of the entropy solution.
The important advantage of shift-variations is that this framework allows to develop
an adjoint calculus for hyperbolic conservation laws by using an \textit{averaged sensitivity equation} which avoids the linearization of (\ref{hypEqIntro}) in the usual way, see \cite{Ulbrich2003}
for further details. The directional derivative of $J$ in (\ref{cF}) in the direction of $\delta u_0$ can then be represented by
\begin{flalign}\label{gradient}
\partial_{u_0}J(y(T,\cdot;u_0),y_d) \, \delta u_0(\cdot ) = \int_I p(0,x)\, \delta u_0(x)\, dx,
\end{flalign}
where $p$ is the solution of the adjoint equation
\begin{equation}
\begin{aligned}\label{revSol}
\partial_tp + f'(y)\partial_xp = &\; 0, \quad\quad (t,x) \in \Omega_T, \\
p(T,x) = &\; p^T(x), \ x \in \R .
\end{aligned}
\end{equation}
Here, $p^{T}(x)$ is given by
\begin{equation}
\begin{aligned}\label{averagedEndData}
p^T(x) = \left\{\begin{array}{ll}
\displaystyle\frac{[G(y(T,x),y_d(x))]}{[y(T,x)]}, & x \in X_s , \\[5mm]
\partial_yG(y(T,x),y_d(x)), & \text{otherwise}, \end{array}\right.
\end{aligned}
\end{equation}
where $X_s$ is the set of locations where $y(T,\cdot)$ possesses a shock and $[w(x)] \coloneqq w(x-)-w(x+)$, which naturally incorporates the shock sensitivity.

Equation (\ref{revSol}) is a linear transport equation with, in general, discontinuous coefficients. It  admits multiple solutions, which requires the selection of the correct adjoint state. This is achieved by so-called \textit{reversible solutions} that are defined along generalized characteristics \cite{Dafermos1977}. An illustrative demonstration is given in Example~\ref{exRevCons}. Under suitable technical assumptions and for appropriate end data $p^T$ it can be shown that there exists a unique reversible solution to (\ref{revSol}) that is bounded, $L^{\infty}$-stable, and $TV$-stable
\cite[Theorem 4.2.10 and Corollary 4.2.11]{Ulbrich2001}. In what follows, we will work with formulation (\ref{revSol}) to derive a discrete adjoint WENO3 method.

\begin{example}\label{exRevCons}
Let $f(y)= \frac{1}{2} y^2$, $u_0(x)=-\text{sign}(x)$, $T=0.5$, and $y_d(x)=0$ with $x \in \R$. It is well-known that the unique entropy solution is given by $y(t,x)=-\text{sign}(x)$, $t \in [0,T]$, and
hence we have
\begin{flalign}
p^{T}(x)= \left\{\begin{array}{ll} 0, & x = 0, \\
         -\text{sign}(x), & x \neq 0. \end{array}\right.
\end{flalign}
The area that is not occupied by the classical characteristics is called
{\normalfont shock funnel}. It is represented by the grey-coloured triangle
in Fig.~\ref{fig:M1}. In this region, $p$ takes the constant value zero. The
adjoint remains constant along the classical backwards characteristics outside
this region. Hence, the reversible solution $p$ is given by
\begin{flalign}\label{exAdjointt}
p(0,x) = \begin{cases}
\quad 1 & ,  x<-\frac{1}{2}, \\[2mm]
\quad 0 & , -\frac{1}{2} \leq  x \leq \frac{1}{2}, \\[2mm]
\,- 1 & ,  \frac{1}{2}<x.
\end{cases}
\end{flalign}

\begin{figure}[ht]
\centering
\begin{tikzpicture}[scale=4.5]
    \draw [<->,thick] (0,1) node (yaxis) [above] {$t$}
        |- (-1.25,0)--(1.25,0) node (xaxis) [right] {$x$};
    \foreach \x/\xtext in {-1, -0.5, 0, 0.5, 1}
       \draw[shift={(\x,0)}] (0pt,1pt) -- (0pt,-1pt) node[below] {$\xtext$};
    \draw[->, thick] (0,0.5) coordinate (b_1) -- (0.5,0) coordinate (b_2);
    \draw[->, thick] (0,0.5) coordinate (c_1) -- (-0.5,0) coordinate (c_2);
    \draw[dashed] (-1.25,0.5) coordinate (a_1) -- (1.25,0.5) coordinate (a_2);
    \draw[->, thick] (0.25,0.5) coordinate (c_1) -- (0.75,0) coordinate (c_2);
    \draw[->, thick] (0.5,0.5) coordinate (c_1) -- (1,0) coordinate (c_2);
    \draw[->, thick] (-0.25,0.5) coordinate (c_1) -- (-0.75,0) coordinate (c_2);
    \draw[->, thick] (-0.5,0.5) coordinate (c_1) -- (-1,0) coordinate (c_2);
    \draw[decoration={brace,mirror,raise=5pt},decorate]
  (0,0.5) -- node[above=9pt] {$p(0.5,x)=1$} (-1,0.5);
   \draw[decoration={brace,raise=5pt},decorate]
    (0,0.5) --  node[above=9pt] {$p(0.5,x)=-1$} (1,0.5);
    \draw[fill=gray]  (-0.5,0) -- (0,0.5) -- (0.5,0) -- cycle;
    \node[draw,align=left] at (0,0.2) {p(t,x)=0};
    \coordinate (c) at (0,0.5);
    \fill[black, label={60:$0$}] (c) circle (1pt);
\end{tikzpicture}
\parbox{12cm}{
\caption{Construction of the reversible solution $p(t,x)$ from the end data $p^T(x)$ at $T=0.5$.
The shock funnel region is accentuated as grey-colored triangle.}
\label{fig:M1}
}
\end{figure}
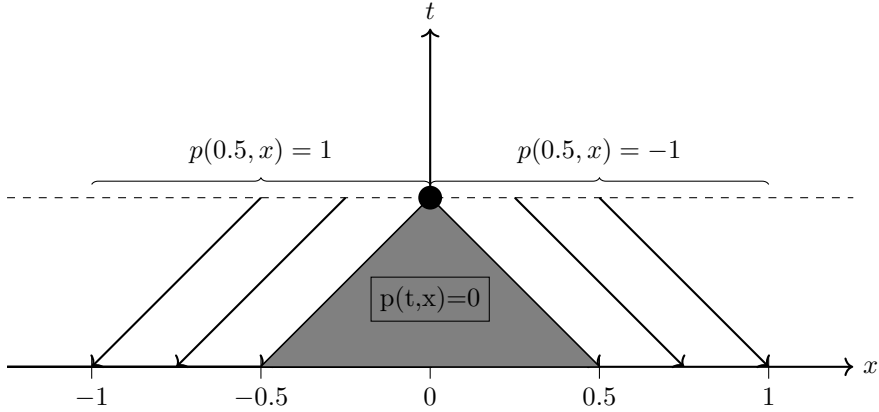

\end{example}

\section{Discrete Adjoint WENO3 Method}
In order to discretize (\ref{hypEqIntro}) in space, we now consider solutions with compact support
$[a,b]$ for the entire time interval $[0,T]$. Thus, using as many points outside as needed and the
compactness of the solution, the implementation of the zero boundary condition does not bear any
difficulty. The interval $[a,b]$ is partitioned into subintervals $[x_{j-1/2},x_{j+1/2}]$ of the same
size $\Delta x$ and with midpoints $x_{j}$ for $j=1,\ldots,N$.
Setting $\bm{u}_{0}:=(u_0(x_1),\ldots,u_0(x_N))\T$ and defining spatial approximations
$\bm{y}(t):=(y_1(t),....,y_N(t))\T$ with $y_j(t) \approx y(t,x_j)$, a spatial semi-discretization of
(\ref{hypEqIntro}) reads
\begin{equation}\label{semiWENO3}
\bm{y}'(t) = -F_{\Delta x} (\bm{y}(t)), \quad
\bm{y}(0) = \bm{u}_{0} \in \R^N,
\end{equation}
where the nonlinear operator $F_{\Delta x}: \R^N \to \R^N$ represents the discretization of
$\partial_xf(y)$. We choose a conservative finite difference
\begin{flalign}\label{spatialDis}
(F_{\Delta x}(\bm{y}(t)))_j = \frac{1}{\Delta x}
\Bigl(\hat{f}_{j+\frac{1}{2}}- \hat{f}_{j-\frac{1}{2}} \Bigr),
\end{flalign}
where $\hat{f}_{j+1/2}: \R^m \to \R$ denotes the numerical flux at $x_{j+1/2}$,
which is (at least) a Lipschitz continuous function of $m$ neighboring values $y_i(t)$.
In order to avoid the convergence of the scheme towards entropy violating solutions,
we apply a global flux splitting
\begin{flalign}\label{fs}
\hat{f}_{j+\frac{1}{2}} = \hat{f}^{+}_{j+\frac{1}{2}} + \hat{f}^{-}_{j+\frac{1}{2}}.
\end{flalign}
Using the simple Lax-Friedrichs splitting $f^{\pm}(y)=(f(y) \pm \alpha y)/2$ with
$\alpha \coloneqq \max_{u} |f'(u)|$ yields the desired properties
$(f^+)'(y)\ge 0$ and $(f^-)'(y)\le 0$. Then, the numerical flux functions of
the WENO3 method \cite{Shu1998} are defined by
\begin{flalign}
\label{posFlux}
\hat{f}^+_{j+\frac{1}{2}}(y_{j-1},y_j,y_{j+1}) \coloneqq &\;
\omega_1^+ \Bigl(-\frac{1}{2} f^+(y_{j-1}) + \frac{3}{2} f^+(y_j) \Bigr)
+ \omega_2^+ \Bigl(\frac{1}{2} f^+(y_j) + \frac{1}{2} f^+(y_{j+1}) \Bigr),\\[2mm]
\label{negFlux}
\hat{f}^-_{j-\frac{1}{2}}(y_{j-1},y_j,y_{j+1}) \coloneqq &\;
\omega_2^- \Bigl(-\frac{1}{2} f^-(y_{j+1}) + \frac{3}{2} f^-(y_j) \Bigr)
+ \omega_1^- \Bigl(\frac{1}{2} f^-(y_{j}) +  \frac{1}{2} f^-(y_{j-1}) \Bigr),
\end{flalign}
where the weights are
\begin{flalign}\label{weights}
\omega_m^\pm =\frac{\tilde{\omega}_m^\pm}{\sum_{i=1,2} \tilde{\omega}_i^\pm},
\quad \tilde{\omega}_m^\pm = \frac{\gamma_m^\pm}{(\varepsilon + \beta_m^\pm)^2}, \quad m=1,2.
\end{flalign}
The smoothness indicators are given by
\begin{flalign}\label{smoothind}
\beta_1^\pm=(f^\pm(y_j)-f^\pm(y_{j-1}))^2, \ \beta_2^\pm= (f^\pm(y_{j+1})-f^\pm(y_j))^2
\end{flalign}
and the linear weights are set to $\gamma_1^+\!=\!\gamma_2^-\!=\!1/3,\;\gamma_1^-\!=\!\gamma_2^+\!=\!2/3$.
Note that $0<\varepsilon \ll 1$ is chosen in order to avoid that the denominator becomes
zero. It is set to $\varepsilon\!=\!10^{-6}$ in our numerical calculations. We would like to
emphasize the observation that by construction the numerical fluxes $\hat{f}^\pm$ have the same smoothness
dependency on its arguments as that of the physical flux function $f(y)$.

Next we will derive the associated adjoint WENO3 scheme. Let $f\in C^2(\R)$, i.e.,
there exists the Fr\'{e}chet derivative of $F_{\Delta x}$ defined in (\ref{spatialDis}).
The continuous optimal control problem is approximated by
\begin{flalign}\label{discproblem}
\bm{u}_0^{min} = \argmin_{\bm{u}_0 \in U_{ad}} \sum_{j=1}^{N} G(y_j(T),y_d(x_j)),
\end{flalign}
where $U_{ad}=\{\bm{u}\in\R_N:\,\text{TV}(\bm{u})\le C\}$ is the discrete admissible set.
Then, applying the common Lagrangian approach in $\R^N$ with multipliers
$\bm{p}(t)=(p_1(t),\ldots,p_N(t))\T$, the adjoint equation to
(\ref{semiWENO3}) reads
\begin{equation}\label{adjsemiWENO3}
\bm{p}'(t) = \nabla_{\bm{y}}F_{\Delta x} (\bm{y}(t))\T\,\bm{p}(t), \quad
\bm{p}(T) = \left(\partial_yG(y_j(T),y_d(x_j))\right)_{j=1,\ldots,N},
\end{equation}
where $\nabla_{\bm{y}}F_{\Delta x}$ is the Fr\'{e}chet derivative of $F_{\Delta x}$ and
gradients are treated as row vectors. The initial condition (the adjoint equation works
backwards in time) is the discrete counterpart to (\ref{averagedEndData}). Observe that the
interior boundary condition does not appear here. A short calculation yields the
componentwise description
\begin{equation}\label{adjsemiWENO3comp}
p_j'(t) = \frac{1}{\Delta x}\sum_{i=-2}^{2} \partial_{y_j}L_{i,j}(\bm{y}(t))\,p_{j+i}(t),
\quad j=1,\ldots,N,
\end{equation}
with the coefficients
\begin{equation}\label{coeffLij}
\begin{aligned}
L_{-2,j}(\bm{y}) = &\; \hat{f}^{-}_{j-3/2}, \\[1mm]
L_{-1,j}(\bm{y}) = &\;
 \hat{f}^{+}_{j-1/2} + \hat{f}^{-}_{j-1/2} - \hat{f}^{-}_{j-3/2}, \\[1mm]
L_{0,j}(\bm{y}) = &\;
 \hat{f}^{+}_{j+1/2} + \hat{f}^{-}_{j+1/2} - \hat{f}^{+}_{j-1/2} - \hat{f}^{-}_{j-1/2}, \\[1mm]
L_{1,j}(\bm{y}) = &\;
 \hat{f}^{+}_{j+3/2} - \hat{f}^{+}_{j+1/2} - \hat{f}^{-}_{j+1/2}, \\[1mm]
L_{2,j}(\bm{y}) = &\; -\hat{f}^{+}_{j+3/2} .
\end{aligned}
\end{equation}
The indices of the numerical flux functions are directly related to their arguments,
e.g. $\hat{f}^{+}_{j+3/2}(y_j,y_{j+1},y_{j+2})$ due to (\ref{posFlux}).
For later use, we note that $\sum_{i=-2,\ldots,2}L_{i,j}(\bm{y})=0$.

We will now study the consistency order of the adjoint WENO3 scheme, i.e., how accurate
does the semi-discretization (\ref{adjsemiWENO3}) approximate the continuous adjoint
equation (\ref{revSol}) in the case of smooth solutions. Inserting exact solution values
$p(t,x_j)$ and $y(t,x_j)$ (still denoted by $y_j$ to simplify notation) in the semi-discrete scheme
(\ref{adjsemiWENO3comp}) gives the residual-type local spatial errors
\begin{equation}
r_j(t) = \partial_t p(t,x_j) -
\frac{1}{\Delta x}\sum_{i=-2}^{2} \partial_{y_j}L_{i,j}(\bm{y}(t))\,p(t,x_{j+i}).
\end{equation}
Taylor expansion around $x_j$ yields
\begin{equation}\label{resAdjWENO3}
r_j(t) = \partial_t p(t,x_j) - \sum_{k\ge 0} \Delta x^k
\frac{1}{(k+1)!}\sum_{i=-2}^{2} i^{k+1}\,\partial_{y_j}L_{i,j}(\bm{y}(t))
\,\partial_x^{k+1}p(t,x_j),
\end{equation}
where we have already used that the sum of the $L_{i,j}$ disappears. The method
is said to have adjoint consistency order $q$ if $r_j(t)=\Oh(\Delta x^q)$.
In what follows, we will show that the adjoint WENO3 scheme satisfies all conditions
for order $q=3$.

First, we have to calculate $\partial_{y_j}L_{i,j}$, i.e., particularly the derivatives of
the numerical flux functions defined in (\ref{posFlux}), (\ref{negFlux}). Since
$\omega_1^\pm+\omega_2^\pm=1$ for all $\bm{y}(t)$, we deduce $\partial_{y_k}\omega_1^\pm=-\partial_{y_k}\omega_2^\pm$. Introducing the notation
\begin{flalign}
\bar{f}_j^{\pm}(y_{j-1},y_j,y_{j+1}) \coloneqq \mp \frac{1}{2}f^+(y_{j-1}) \pm f^+(y_{j}) \mp \frac{1}{2}f^+(y_{j+1}),
\end{flalign}
we find
\begin{equation}\label{derivFplus}
\begin{aligned}
\partial_{y_{j-1}} \hat{f}^+_{j+1/2} =&\;
\partial_{y_{j-1}} \omega_1^+ \bar{f}_j^+ - \frac12 (f^+)'(y_{j-1})w_1^+, \\[1mm]
\partial_{y_{j}} \hat{f}^+_{j+1/2} =&\;
\partial_{y_{j}} \omega_1^+ \bar{f}_j^+ + (f^+)'(y_j)
\left( \frac32 w_1^+ + \frac12 w_2^+ \right) , \\[1mm]
\partial_{y_{j+1}} \hat{f}^+_{j+1/2} =&\;
\partial_{y_{j+1}} \omega_1^+ \bar{f}_j^+ + \frac12 (f^+)'(y_{j+1})w_2^+ ,
\end{aligned}
\end{equation}
and
\begin{equation}\label{derivFminus}
\begin{aligned}
\partial_{y_{j-1}} \hat{f}^-_{j-1/2} =&\;
\partial_{y_{j-1}} \omega_1^- \bar{f}_j^- + \frac12 (f^-)'(y_{j-1})w_1^-, \\[1mm]
\partial_{y_{j}} \hat{f}^-_{j-1/2} =&\;
\partial_{y_{j}} \omega_1^- \bar{f}_j^- + (f^-)'(y_j)
\left( \frac12 w_1^- + \frac32 w_2^- \right) , \\[1mm]
\partial_{y_{j+1}} \hat{f}^-_{j-1/2} =&\;
\partial_{y_{j+1}} \omega_1^- \bar{f}_j^- - \frac12 (f^-)'(y_{j+1})w_2^- .
\end{aligned}
\end{equation}

We have the following three lemmata.
\begin{lemma}\label{lem:01}
Suppose $f(y),\;y(t,\cdot)\in C^2(\R)$. Then
\begin{equation}
\partial_{y_k} \omega_1^\pm \bar{f}^{\pm}_j(y_{j-1},y_j,y_{j+1})=
\Oh\left(\Delta x^3\right),\quad k=j-1,j,j+1.
\end{equation}
\end{lemma}
\begin{proof} Taylor expansion gives $\bar{f}^{\pm}_j=\Oh(\Delta x^2)$.
It remains to show that $\partial_{y_k} \omega_1^\pm=\Oh(\Delta x)$. Indeed, we
have
\begin{equation}
\partial_{y_k} \omega_1^\pm =
\frac{\tilde{\omega}_2^\pm}{(\tilde{\omega}_1^\pm + \tilde{\omega}_2^\pm)^2}
\,\partial_{y_k} \tilde{\omega}_1^\pm
-
\frac{\tilde{\omega}_1^\pm}{(\tilde{\omega}_1^\pm + \tilde{\omega}_2^\pm)^2}
\,\partial_{y_k} \tilde{\omega}_2^\pm\,.
\end{equation}
The two quotients are bounded by $(\tilde{\omega}_i^\pm)^{-1}=\Oh(\varepsilon^2)$,
$i=2,1$, respectively, for $\Delta x\rightarrow 0$. Taylor expansions of the derivatives
$\partial_{y_k} \tilde{\omega}_i^\pm=-2\gamma_i^\pm(\varepsilon+\beta_i^\pm)^{-3}
\partial_{y_k}\beta_i^\pm$, $i=1,2$, show $\Oh(\Delta x)$ for these terms and
therefore also for $\partial_{y_k} \omega_1^\pm$.
\end{proof}
\begin{lemma}\label{lem:02}
Let $\{x_{j-1},x_{j},x_{j+1}\}$ and $\{x_{j},x_{j+1},x_{j+2}\}$ be two neighboring
stencils and $w_{i,j}^\pm$, $w_{i,j+1}^\pm$, $i=1,2,$ the corresponding weights.
Suppose $f(y),\;y(t,\cdot)\in C^3(\R)$. Then
\begin{equation}
w_{i,j+1}^\pm - w_{i,j}^\pm = \Oh\left(\Delta x^4\right),\quad i=1,2.
\end{equation}
\end{lemma}
\begin{proof}
We consider the weights $w_{1,j}^+$ and $w_{1,j+1}^+$. Analogous calculations can
be done for the other cases. We set $h(x):=f^+(y(x))$ and define $h_j:=f^+(y(x_j))$.
Then
\begin{equation}
\begin{aligned}
w_{1,j+1}^+ - w_{1,j}^+ = & \;
\frac{\tilde{\omega}_{1,j+1}^+}{\tilde{\omega}_{1,j+1}^+ + \tilde{\omega}_{2,j+1}^+}
- \frac{\tilde{\omega}_{1,j}^+}{\tilde{\omega}_{1,j}^+ + \tilde{\omega}_{2,j}^+}\\
= & \;
\frac{
\frac{\tilde{\omega}_{2,j}^+}{\tilde{\omega}_{1,j}^+}-
\frac{\tilde{\omega}_{2,j+1}^+}{\tilde{\omega}_{1,j+1}^+}
}
{
\left( 1+ \frac{\tilde{\omega}_{2,j}^+}{\tilde{\omega}_{1,j}^+}\right)\,
\left( 1+ \frac{\tilde{\omega}_{2,j+1}^+}{\tilde{\omega}_{1,j+1}^+}\right)
}\,.
\end{aligned}
\end{equation}
Due to the strict positivity of the weights, it remains to study the asymptotic
behaviour of the nominator. Using the definitions (\ref{weights}) and (\ref{smoothind}),
we have
\begin{equation}\label{diffw}
D_{\tilde{w}}:=\frac{\tilde{\omega}_{2,j}^+}{\tilde{\omega}_{1,j}^+}-
\frac{\tilde{\omega}_{2,j+1}^+}{\tilde{\omega}_{1,j+1}^+} =
\frac{\gamma_2^+}{\gamma_1^+}\,
\frac{(\varepsilon + \beta_{1,j}^+)^2(\varepsilon + \beta_{2,j+1}^+)^2
- (\varepsilon + \beta_{1,j+1}^+)^2(\varepsilon + \beta_{2,j}^+)^2}
{(\varepsilon + \beta_{2,j}^+)^2(\varepsilon + \beta_{2,j+1}^+)^2}
\end{equation}
with the smoothness indicators
\begin{equation}
\begin{aligned}
\beta_{1,j}^+ = &\; (h_{j}-h_{j-1})^2,\quad \beta_{2,j}^+ = (h_{j+1}-h_{j})^2\,,\\
\beta_{1,j+1}^+ = &\; (h_{j+1}-h_{j})^2,\quad \beta_{2,j+1}^+ = (h_{j+2}-h_{j+1})^2\,.
\end{aligned}
\end{equation}
Taylor expansion at $x_j$ yields in (\ref{diffw})
\begin{equation}
D_{\tilde{w}} = \frac{\gamma_2^+}{\varepsilon^4\gamma_1^+}
\left(
4\varepsilon^3 \Delta x^4 \left( (h''_j)^2+h'_jh'''_j\right) +
\Oh(\varepsilon^3 \Delta x^5) \right)\,,
\end{equation}
which shows the assertion.
\end{proof}
\begin{lemma}\label{lem:03}
Suppose $f(y),\;y(t,\cdot)\in C^2(\R)$. Then
\begin{equation}
\omega_1^\pm - \gamma_1^\pm = \Oh\left(\Delta x^3\right)\,.
\end{equation}
\end{lemma}
\begin{proof}
We first consider $\omega_1^+ - \gamma_1^+$. The difference can be expressed by
\begin{equation}
\omega_1^+ - \gamma_1^+ =
\frac{\tilde{\omega}_1^+-\gamma_1^+ (\tilde{\omega}_1^+ + \tilde{\omega}_2^+)}
{\tilde{\omega}_1^+ + \tilde{\omega}_2^+}.
\end{equation}
The denominator is bounded from below by
$(\gamma_1^+ + \gamma_2^+)\varepsilon^{-2}=\varepsilon^{-2}>0$. Further,
we deduce for the nominator
\begin{equation}
N_{\tilde{\omega}}:=\tilde{\omega}_1^+-\gamma_1^+ (\tilde{\omega}_1^+ + \tilde{\omega}_2^+) =
\gamma_1^+\gamma_2^+\;\frac{(\varepsilon+\beta_2^+)^2-(\varepsilon+\beta_1^+)^2}
{(\varepsilon+\beta_1^+)^2\,(\varepsilon+\beta_2^+)^2}\,.
\end{equation}
Let $h(x):=f^+(y(x))$ and define $h_j:=f^+(y(x_j))$. Inserting the smoothness
indicators $\beta_1^+=(h_j-h_{j-1})^2$ and $\beta_2^+=(h_{j+1}-h_{j})^2$,
Taylor expansion at $x_j$ yields
\begin{equation}
N_{\tilde{\omega}} = \frac{\gamma_1^+\gamma_2^+}{\varepsilon^4}
\left( 4\varepsilon\Delta x^3 h'_jh''_j + \Oh\left(\Delta x^5\right) \right)\,.
\end{equation}
Putting this together with the bound for the denominator stated above gives
$\omega_1^+ - \gamma_1^+ = \Oh(\Delta x^3/\varepsilon)$, from which we can
conclude the proof. The same arguments apply to the second
difference $\omega_1^- - \gamma_1^-$.
\end{proof}
We are now ready to state the main result of this section.
\begin{theorem}\label{theo:01}
Let $f(y),\;y(t,\cdot)\in C^3(\R)$ and $p(t,\cdot)\in C^4(\R)$.
Then the adjoint WENO3 scheme (\ref{adjsemiWENO3}) is adjoint consistent
of order three, i.e., $r_j(t)=\Oh(\Delta x^3)$ in (\ref{resAdjWENO3}).
\end{theorem}
\begin{proof}
Let us define
$d_k:=\sum_{i=-2,\ldots,2}i^{k+1}\partial_{y_j}L_{i,j}(\bm{y}(t))$, $k=0,1,2,$
and denote by $w_{i,m}^\pm$ the weights that correspond to the stencils
$\{x_{m-1},x_{m},x_{m+1}\}$, $m=j-1,j,j+1$. From (\ref{coeffLij}), we calculate
\begin{equation}
\begin{aligned}
d_0 &\; = -\partial_{y_j} \left( 2L_{-2,j}(\bm{y}(t))+
L_{-1,j}(\bm{y}(t))-L_{1,j}(\bm{y}(t))-
2L_{2,j}(\bm{y}(t))\right) \\[1mm]
    &\; = -\partial_{y_j}\left(
\hat{f}^-_{j-3/2} + \hat{f}^+_{j-1/2} + \hat{f}^-_{j-1/2} +
\hat{f}^+_{j+1/2} + \hat{f}^-_{j+1/2} + \hat{f}^+_{j+3/2}\right),
\end{aligned}
\end{equation}
which gives by using (\ref{derivFplus}), (\ref{derivFminus}) for different stencils
and Lemma \ref{lem:01} for all terms $\partial_{y_j} \omega_1^\pm \bar{f}^{\pm}_m$
with $m=j-1,j,j+1$,
\begin{equation}
\begin{aligned}
d_0  = &\; -\left(
\frac12w_{1,j}^{-}+\frac12w_{1,j+1}^{-}-\frac12w_{2,j-1}^{-}+\frac32w_{2,j}^{-}
\right)(f^-)'(y_j) \\[2mm]
&\; -\left(
\frac32w_{1,j}^{+}-\frac12w_{1,j+1}^{+}+\frac12w_{2,j-1}^{+}+\frac12w_{2,j}^{+}
\right)(f^+)'(y_j) + \Oh(\Delta x^3).
\end{aligned}
\end{equation}
Eventually, Lemma \ref{lem:02} and the property $w_{1,j}^\pm+w_{2,j}^\pm=1$ yields
\begin{equation}
d_0 = - \left( (f^-)'(y_j) + (f^+)'(y_j) \right) + \Oh(\Delta x^3)
= -f'(y_j) + \Oh(\Delta x^3).
\end{equation}
Analogously, we derive
\begin{equation}
\begin{aligned}
d_1  = &\; \left(
\frac12w_{1,j}^{-}-\frac12w_{1,j+1}^{-}-\frac32w_{2,j-1}^{-}+\frac32w_{2,j}^{-}
\right)(f^-)'(y_j) \\[2mm]
&\; + \left(
-\frac32w_{1,j}^{+}+\frac32w_{1,j+1}^{+}+\frac12w_{2,j-1}^{+}-\frac12w_{2,j}^{+}
\right)(f^+)'(y_j) + \Oh(\Delta x^3)
\end{aligned}
\end{equation}
and
\begin{equation}
\begin{aligned}
d_2  = &\; \left(
-\frac12w_{1,j}^{-}-\frac12w_{1,j+1}^{-}+\frac72w_{2,j-1}^{-}-\frac32w_{2,j}^{-}
\right)(f^-)'(y_j) \\[2mm]
&\; + \left(
-\frac32w_{1,j}^{+}+\frac72w_{1,j+1}^{+}-\frac12w_{2,j-1}^{+}-\frac12w_{2,j}^{+}
\right)(f^+)'(y_j) + \Oh(\Delta x^3)\,.
\end{aligned}
\end{equation}
Lemma \ref{lem:02} directly shows that $d_1=\Oh(\Delta x^3)$. Using $w_1^\pm+w_2^\pm=1$ and
again Lemma \ref{lem:02}, the linear combinations of the weights in $d_2$ can be simplified to
$2-3w_{1,j}^-$ and $3w_{1,j}^+-1$ up to order $\Oh(\Delta x^4)$, respectively.
Applying now Lemma \ref{lem:03} with $\gamma_1^+=1/3$ and $\gamma_1^-=2/3$ to these expressions
gives $d_2=\Oh(\Delta x^3)$.

In a last step, we use the asymptotic expressions for $d_i$, $i=0,1,2,$
to calculate the residual-type local spatial error
\begin{equation}
\begin{aligned}
r_j(t) = &\; \partial_tp(t,x_j)-\sum_{k=0}^2\Delta x^k\frac{1}{(k+1)!}
d_k\,\partial_{x}^{k+1}p(t,x_j) + \Oh\left( \Delta x^3\right)\\[1mm]
= &\; \partial_tp(t,x_j) + f'(y(t,x_j))\,\partial_xp(t,x_j) + \Oh\left( \Delta x^3\right)
= \Oh\left( \Delta x^3\right)\,.
\end{aligned}
\end{equation}
This concludes the proof.
\end{proof}

\section{Numerical Experiments}
In this section, we will present some numerical examples for Burgers equation, i.e.,
we study problems with the nonlinear flux function $f(y)=\frac12 y^2$ in (\ref{hypEqIntro}).
The first example with smooth initial data and solution is chosen in order to check
to third-order convergence of the discrete adjoint WENO3 method as stated in
Theorem \ref{theo:01}. In the second example, the approximation property of the discrete
adjoint in the case of a shock in the initial solution is investigated and compared to
approximations computed by means of the first-order modified Lax-Friedrichs (LF) and
Engquist-Osher (EO) schemes. These schemes read
\begin{equation}
\begin{aligned}
y_j^0 = &\; u_0(x_j)\,,\\[1mm]
y_j^{n+1} = &\; y_j^n - \frac{\Delta x}{\Delta t}
\left( \hat{f}(y_j^n,y_{j+1}^n) - \hat{f}(y_{j-1}^n,y_{j}^n)\right),\quad n=0,\ldots,n_T-1,
\end{aligned}
\end{equation}
with $y_j^n\approx y(n\Delta t,x_j)$, $n_T\,\Delta t=T$, and
numerical fluxes given by
\begin{equation}
\begin{aligned}
\hat{f}_{LF}(a,b) = &\; \frac12 \left( f(b)+f(a)\right) -
\frac{\gamma}{2}\frac{\Delta x}{\Delta t} \left( b-a \right),
\quad \gamma\in (0,1)\,,\\[2mm]
\hat{f}_{EO}(a,b) = &\; f(0) + \int_{0}^{a}\max(0,f'(s))\,ds +
\int_{0}^{b}\min(0,f'(s))\,ds\,.
\end{aligned}
\end{equation}
Applying a standard Lagrangian approach and discrete adjoint calculus,
the discrete adjoint schemes can be derived from
\cite[Prep. 3.1]{HajianHintermuellerUlbrich2017} as
\begin{equation}
\begin{aligned}
p_j^{n_T} = &\; \partial_yG(y_j^{n_T},y_d(x_j))\,,\\[1mm]
p_j^n = &\; c_{j-1}p_{j-1}^{n+1} + c_{j}p_{j}^{n+1} + c_{j+1}p_{j+1}^{n+1},\quad
n=n_T-1,\ldots,0,
\end{aligned}
\end{equation}
with the coefficients
\begin{equation}
c_{j-1}=\frac{\gamma}{2} -
\frac{\Delta t}{2\Delta x}f'(y_j^{n+1}),\quad
c_j=1-\gamma,\quad
c_{j+1}=\frac{\gamma}{2} +
\frac{\Delta t}{2\Delta x}f'(y_j^{n+1}),\\
\end{equation}
for the LF scheme and
\begin{equation}
\begin{aligned}
c_{j-1} = &\; \frac{\Delta t}{2\Delta x}
\left( |f'(y_j^{n+1})| - f'(y_j^{n+1}) \right),\quad
c_j=1-\frac{\Delta t}{\Delta x} |f'(y_j^{n+1})|,\\[2mm]
c_{j+1} = &\; \frac{\Delta t}{2\Delta x}
\left( |f'(y_j^{n+1})| + f'(y_j^{n+1}) \right),
\end{aligned}
\end{equation}
for the EO scheme. Convergence of these schemes has been
intensively studied in \cite{AguilarSchmittUlbrichMoos2019,Giles2003,GilesUlbrich2010a,GilesUlbrich2010b,Ulbrich2001}.
The choice $\gamma\!=\!1$ leads to the classical LF method. Stability requirements for
the adjoint LF and EO schemes yield the optimal value $\gamma^\star\!=\!0.5$ together
with the CFL-condition $\Delta t \le \gamma^\star\Delta x/\sup |f'(y)|$, see e.g. \cite{HajianHintermuellerUlbrich2017}. Then, both schemes converge for Lipschitz continuous
end data $p^T(x)$ in (\ref{revSol}). The stronger condition
$\Delta t \le \gamma^\star (\Delta x)^{2-q}/\sup |f'(y)|$, $0<q<1$, ensures the convergence
of the modified LF scheme for discontinuous end data, too \cite{GilesUlbrich2010a,GilesUlbrich2010b}. Convergence for slightly modified end data and less
numerical viscosity has been recently studied in \cite{AguilarSchmittUlbrichMoos2019}.

In order to get a fully discrete scheme for WENO3, the differential equation (\ref{semiWENO3}) is
numerically solved by the three-stage third-order strong-stability-preserving
Runge-Kutta method SSPRK3, which offers good stability properties \cite{GottliebKetchesonShu2011,GottliebShuTadmor2001,HajianHintermuellerUlbrich2017,
HintermuellerStrogis2018}. In the Shu-Osher representation, it reads
\begin{equation}
\begin{aligned}
\bm{y}_0^n = &\; \bm{y}^n\,,\\[1mm]
\bm{y}_1^n = &\; \bm{y}_0^n - \Delta t F_{\Delta x}(\bm{y}_0^n)\,,\\[1mm]
\bm{y}_2^n = &\; \frac34 \bm{y}_0^n + \frac14 \bm{y}_1^n -
\frac14 \Delta t F_{\Delta x}(\bm{y}_1^n)\,,\\[1mm]
\bm{y}^{n+1} = &\; \frac14 \bm{y}_0^n + \frac23 \bm{y}_2^n -
\frac23 \Delta t F_{\Delta x}(\bm{y}_2^n)\,,\quad n=0,\ldots,n_T-1.
\end{aligned}
\end{equation}
The corresponding adjoint time discretization has the form
(see e.g. \cite{HajianHintermuellerUlbrich2017})
\begin{equation}
\begin{aligned}
\bm{p}_0^n = &\; \bm{p}^{n+1}\,,\\[1mm]
\bm{p}_1^n = &\; \frac23 \bm{p}_0^n
- \frac23 \Delta t \,\nabla_{\bm{y}} F_{\Delta x}(\bm{y}_2^n)^T\,\bm{p}_0^n\,,\\[1mm]
\bm{p}_2^n = &\; \frac14 \bm{p}_1^n
- \frac14 \Delta t \,\nabla_{\bm{y}} F_{\Delta x}(\bm{y}_1^n)^T\,\bm{p}_1^n\,,\\[1mm]
\bm{p}^n = &\; \frac13 \bm{p}_0^n + \frac34 \bm{p}_1^n + \bm{p}_2^n
- \Delta t \,\nabla_{\bm{y}} F_{\Delta x}(\bm{y}_0^n)^T\,\bm{p}_2^n\,,\quad n=n_T-1,\ldots,0.
\end{aligned}
\end{equation}
We note that the adjoint scheme has only order two, which is the upper barrier for
three-stage third-order SSPRK methods \cite{HajianHintermuellerUlbrich2017}.

In the final experiment, we solve an optimal control problem with a discontinuous
target, proposed in \cite{HajianHintermuellerUlbrich2017}.
The discrete  adjoint $\bm{p}^0$ provides gradient information, which can be directly
used to set up the following algorithm:
\begin{itemize}
\item[0.] Given a control $\bm{u}_0:=\bm{u}^{(j)}$ at iteration $j$.
\item[1.] Compute the discrete adjoint $\bm{p}^0(\bm{u}^{(j)})$ and update
$\bm{u}^{(j+1)}=\bm{u}^{(j)}-\alpha_j\bm{p}^0(\bm{u}^{(j)})$ with $\alpha_j$ such
that Armijo's condition
\[ J(\bm{y}^{n_T}(\bm{u}^{(j+1)}),\bm{y}_d) \le J(\bm{y}^{n_T}(\bm{u}^{(j)}),\bm{y}_d) -
\frac12\alpha_j \| \bm{p}^0(\bm{u}^{(j)}) \|^2_{L^2(I)}\]
is fulfilled. If it is not satisfied, choose $\alpha_j:=0.95\,\alpha_j$ and check
the condition again.
\item[2.] Stop if $|J(\bm{y}^{n_T}(\bm{u}^{(j+1)}),\bm{y}_d)-
J(\bm{y}^{n_T}(\bm{u}^{(j)}),\bm{y}_d)|\le tol$. Otherwise set $j:=j+1$ and proceed
with step 1.
\end{itemize}
In general, taking the adjoint as a decent direction may increase the complexity of the
optimization process due to the production of additional discontinuities
\cite{CastroPalaciosZuazua2008,LecarosZuazua2014,LecarosZuazua2016}. A careful choice of
the initial guess $\bm{u}_0$ can remedy this serious problem. We follow the approach
proposed in \cite{HertyKurganovKurochkin2015} and first solve the conservation
law
\begin{equation}\label{claw-init}
\begin{aligned}
\partial_t z + \partial_x f(z) =&\; 0,\quad (t,x)\in\Omega_T,\\
z(0,x) =&\;y_d(-x),
\end{aligned}
\end{equation}
where $y_d$ is the target given in (\ref{problem}). The initial guess is then chosen
as $u_0=z(T,-x)$. Formally, as pointed out in \cite{HertyKurganovKurochkin2015},
(\ref{claw-init}) is obtained by reverting $t$ and $x$ in (\ref{hypEqIntro}) and taking $y_d$
as initial condition. The advantage of this approach is that it delivers a control whose
entropy solution is close to the target and the location of the discontinuities almost
coincide. Hence, the production of additional discontinuities within each iteration
step is avoided, which improves the performance of the algorithm drastically. We will
exemplify the influence of the choice of the initial guess in our optimal control problem.

\subsection{Order Test for the Discrete Adjoint for Smooth Data}
This section is devoted to numerically verify the third-order convergence of the adjoint
WENO3 scheme. For this purpose, we choose the computational domain
$\Omega_T=(0,0.5]\times [-1.5,1.5]$ and the objective functional
\begin{equation}
J(y(0.5,\cdot;u_0),0) = \frac12\,\int_{-\frac32}^{\frac32} y(0.5,x;u_0)^2\,dx
\end{equation}
with the smooth initial data
\begin{equation}
u_0(x)=\begin{cases}
e^{-\frac{1}{1-x^2}} & ,  |x| < 1, \\
0 & , |x| \geq 1.
\end{cases}
\end{equation}
The exact solution $y(t,x)$ can be directly computed from the method
of characteristics, i.e., $y(t,x)=u_0(x_0(x(t),t))$ with $x_0(x(t),t)$
being the solution of the nonlinear equation $x(t)=x_0+u_0(x_0)\,t$.
A reference solution $y_T\approx y(0.5,x)$ at the final time is computed
by Newton's method with a high tolerance $10^{-14}$.

Since shocks are not present, we find $p^T(x)=y(0.5,x)$ in (\ref{revSol}).
We also note that the characteristics curves of the adjoint problem coincide
with the characteristic curves of the forward problem. Thus, the
corresponding reversible solution $p(0,x)$ at time $t\!=\!0$ is
given by $u_0(x)$, which serves as reference solution for the adjoint.

We use a sequence of spatial meshes with a number of grid points
$N=150\cdot 2^i,\;i=0,\ldots,6,$ and set $\Delta t=0.5\,\Delta x$. In
order to keep the temporal error below $\Oh((\Delta x)^3)$, we apply
the classical fourth-order four-stage explicit Runge-Kutta method (ERK4).
Its adjoint time discretization has also order four \cite{Hager2000} for
smooth solutions and therefore the overall scheme is suitable to check
the order three of the adjoint WENO3 method. We also present results for
the forward WENO3 method to document the error of the approximated starting value
$\bm{p}^{n_T}=\bm{y}^{n_T}$. The $L^\infty$-errors collected in
Tab.~\ref{tab-prob1-smooth} clearly show asymptotic order three of the spatial WENO3
discretization for both forward and adjoint numerical solution.

\begin{table}[ht!]
\begin{center}
\begin{tabular}{|r|c|c|c|c|}\hline
\rule{0mm}{0.4cm}N & $\|\bm{y}_T-\bm{y}^{n_T}\|_\infty$ & rate
& $\|\bm{u}_0-\bm{p}^0\|_\infty$ & rate \\ \hline%
\rule{0mm}{0.4cm} $150$ & $2.00e\!-\!3$ & & $7.39e\!-\!3$ & \\
\rule{0mm}{0.4cm} $300$ & $3.25e\!-\!4$ & $2.63$ & $9.37e\!-\!4$ & $2.98$ \\
\rule{0mm}{0.4cm} $600$ & $2.64e\!-\!5$ & $3.62$ & $7.14e\!-\!5$ & $3.71$ \\
\rule{0mm}{0.4cm}$1200$ & $2.16e\!-\!6$ & $3.62$ & $4.30e\!-\!6$ & $4.05$ \\
\rule{0mm}{0.4cm}$2400$ & $2.76e\!-\!7$ & $2.97$ & $5.49e\!-\!7$ & $2.97$ \\
\rule{0mm}{0.4cm}$4800$ & $3.46e\!-\!8$ & $2.99$ & $6.92e\!-\!8$ & $2.99$ \\
\rule{0mm}{0.4cm}$9600$ & $4.33e\!-\!9$ & $3.00$ & $8.66e\!-\!9$ & $3.00$ \\
\hline
\end{tabular}
\vspace{-2mm}
\parbox{13cm}{
\caption{Burgers problem with smooth initial data and smooth solution:
$L^\infty$-error of the forward solution $\|\bm{y}_T-\bm{y}^{n_T}\|_\infty$
at the final time $T=0.5$ and adjoint solution $\|\bm{u}_0-\bm{p}^0\|_\infty$
at time $t=0$ for a sequence of spatial meshes with $N=150,300,\ldots,9600$
grid points. The convergence rates are computed from $\ln(E_N/E_{2N})/\ln(2)$,
where $E_N$ stands for the corresponding error.}
\label{tab-prob1-smooth}}
\end{center}
\end{table}

\subsection{Approximation of the Discrete Adjoint in the Case of Shocks}
We now consider discontinuous solutions with shocks. Our test case
is taken from Example \ref{exRevCons} with computational domain
$\Omega_T=(0,0.5]\times [-1,1]$. The reversible solution $p(0,x)$ at $t\!=\!0$
is given by (\ref{exAdjointt}). We apply the above described forward and
adjoint LF, EO, and
WENO3 schemes with $\Delta x=0.01,\;0.002$, and $\Delta t=0.25\,\Delta x$.
The corresponding numerical approximations $\bm{p}^0$ are shown in
Fig.~\ref{pic:shock1}.

\begin{figure}[ht]
\centering
\definecolor{mycolor1}{rgb}{0.00000,0.44700,0.74100}%
\definecolor{mycolor2}{rgb}{0.85000,0.32500,0.09800}%
\definecolor{mycolor3}{rgb}{0.92900,0.69400,0.12500}%
\definecolor{mycolor4}{rgb}{0.49400,0.18400,0.55600}%
\input{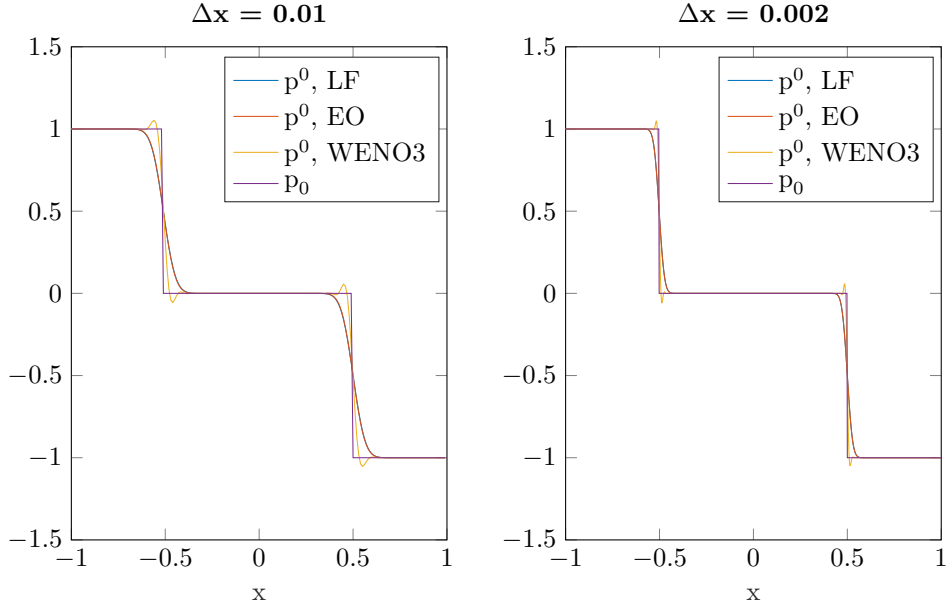}
\parbox{13cm}{
\caption{Burgers problem with discontinuous initial and final solution
taken from Example \ref{exRevCons}.
Numerical approximations $\bm{p}^0$ to the reversible solution
$p_0:=p(0,x)$ given in
(\ref{exAdjointt}) for the adjoint Lax-Friedrichs (LF), Engquist-Osher (EO)
and WENO3 scheme applied with $\Delta x=0.01$ (left), $\Delta x=0.002$
(right), and $\Delta t=0.25\,\Delta x$.}
\label{pic:shock1}}
\end{figure}

The first-order LF and EO schemes smear out the discontinuities, but deliver
$L^\infty$-stable approximations and thus respect the analytical property of
the adjoint. In the spirit of WENO schemes, the adjoint WENO3 delivers a quite
sharp resolution of the shocks at the price of bounded over- and undershoots
of around $5\%$. In Tab.~\ref{tab-prob2-shock1}, we plot the $L^\infty$-error
in the shock funnel for $x\in[-0.3,0.3]$. All schemes converge quite rapidly.
Note that convergence in the shock funnel is not always achieved since the
interior boundary condition at shock positions as given in (\ref{averagedEndData})
does not appear on the discrete level, see the discussions in
\cite{AguilarSchmittUlbrichMoos2019,GilesUlbrich2010a,GilesUlbrich2010b}.

\begin{table}[ht!]
\begin{center}
\begin{tabular}{|l|l|l|l|}\hline
\rule{0mm}{0.4cm} $\Delta x$ & LF & EO & WENO3 \\ \hline
\rule{0mm}{0.4cm} $0.01$  & $4.91e\!-\!05$ & $2.45e\!-\!05$ & $3.92e\!-\!05$\\
\rule{0mm}{0.4cm} $0.002$ & $2.26e\!-\!17$ & $5.79e\!-\!20$ & $6.51e\!-\!16$\\
\hline
\end{tabular}
\vspace{-2mm}
\parbox{13cm}{
\caption{Burgers problem with discontinuous initial and final solution
taken from Example \ref{exRevCons}. $L^\infty$-error of the adjoint solution $\|\bm{p}_0-\bm{p}^0\|_\infty$ at time $t=0$ in the shock funnel
$x\in[-0.3,0.3]$ for $\Delta x=0.01,\;0.002$.}
\label{tab-prob2-shock1}}
\end{center}
\end{table}

\subsection{Optimal Control Problem with Discontinuous Target}
We consider the optimal control problem (\ref{problem})
with the objective functional \cite{HajianHintermuellerUlbrich2017}
\begin{equation}\label{prob3_obj}
J(y(0.5,\cdot;u_0),y_d(x)) = \frac12\,\int_{-1}^{1} |y(0.5,x;u_0)-y_d(x)|^2\,dx
\end{equation}
and the discontinuous target $y_d$ defined by
\begin{equation}
y_d(x)=\begin{cases}
2x-\frac12 & , \frac14\le x\le\frac34, \\
0 & , \text{otherwise}.
\end{cases}
\end{equation}
The optimal control $u_0^\star$, which serves as a reference solution, is
\begin{equation}
u_0^\star(x)=\begin{cases}
-2x+\frac32 & , \frac14\le x\le\frac34, \\
0 & , \text{otherwise}.
\end{cases}
\end{equation}
We will present results for two mesh sizes $\Delta x=0.005,\;0.002$, and
time steps $\Delta t=0.25\,\Delta x$. The initial guess for the control
is computed from (\ref{claw-init}) with the individual method under
consideration. For WENO3 and the coarser mesh size, it is shown
in Fig.~\ref{pic:ctr1} together with the corresponding state solution.

In Fig.~\ref{pic:ctr2}, the results of the gradient based optimization
procedure described above for tolerances $tol_1=10^{-5}$, $tol_2=10^{-7}$,
and mesh size $\Delta x=0.005$ for the adjoint WENO3 method are plotted.
We can conclude that the adjoint WENO3 method allows to recover the initial
data together with the final
state solution adequately. The shock of the target is sharply resolved
and the rarefaction of the initial data is also recovered. In order to compare
these results with those obtained from the LF and EO schemes, we perform $50$
iterations of the optimization algorithm for both mesh sizes. The calculated
optimal controls and their corresponding final state solutions are collected
in Fig.~\ref{pic:ctr3}. The adjoint WENO3 method resolves the shock sharply.
In contrast, the LF method is too diffusive and only provides an unsatisfactory
shock resolution. The numerical artifacts around the shocks are huge. The
optimized final state solution obtained by the EO scheme possesses very small
numerical artifacts, but the shock is less sharply resolved and the spike of
the target is slightly smeared out.
In Tab.~\ref{tab-prob3-objfunc}, we depict the iteration history for all runs of the optimization.
In every case, the LF method performs poorer than the others. In terms of a low cost
functional, the adjoint WENO3 method performs best. We also see the influence of the initial
guess on the performance of the algorithm. This is due to the fact that the use of
$\bm{u}_0=0$ as starting control value produces artificial discontinuities within
each iteration step.

\begin{figure}[ht]
\centering
\definecolor{mycolor1}{rgb}{0.00000,0.44700,0.74100}%
\definecolor{mycolor2}{rgb}{0.85000,0.32500,0.09800}%
\hspace*{-1cm}\input{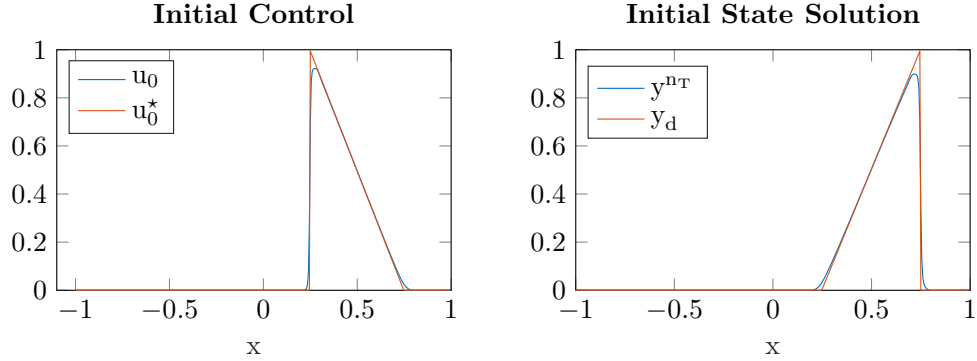}
\parbox{13cm}{
\caption{Optimal control problem. Initial control $\bm{u}_0$ and
optimal control $u_0^\star$ (left), initial state solution $\bm{y}^{n_T}$
at $T=0.5$ and target $y_d$ (right), computed with the WENO3 scheme
and mesh size $\Delta x=0.005$.}
\label{pic:ctr1}}
\end{figure}

\vspace{0.3cm}

\begin{figure}[h!]
\centering
\definecolor{mycolor1}{rgb}{0.00000,0.44700,0.74100}%
\definecolor{mycolor2}{rgb}{0.85000,0.32500,0.09800}%
\hspace*{-1cm}\input{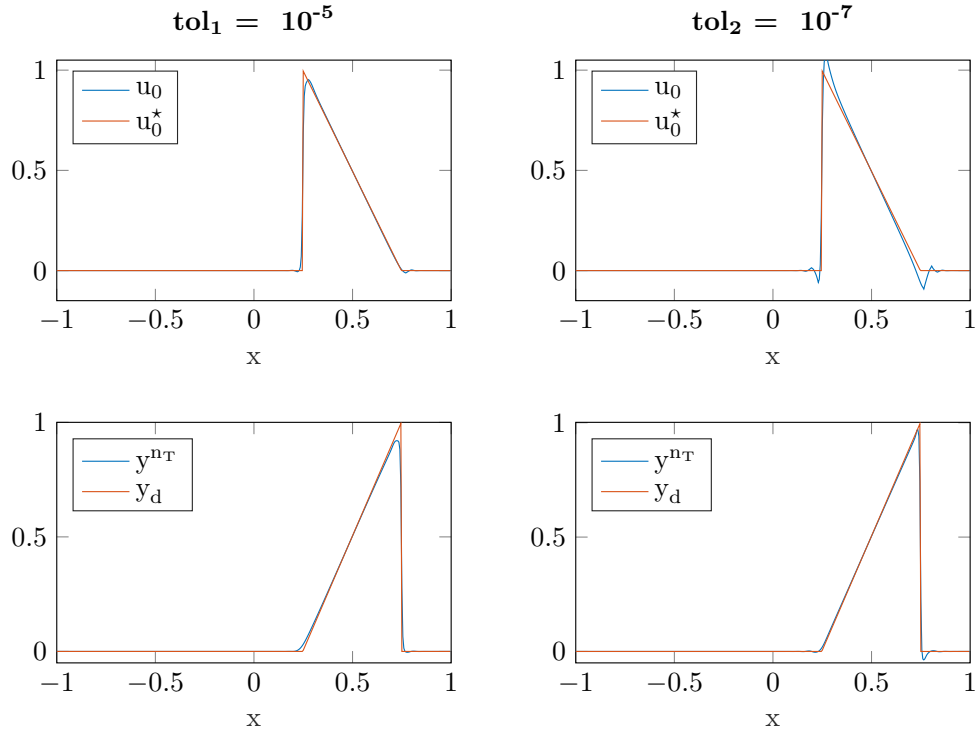}
\vspace{-5mm}
\parbox{13cm}{
\caption{Optimal Control Problem. Optimal control $u_{0}^\star$
and target $y_d$, numerically computed optimal control $\bm{u}_0$ and
corresponding state solution $\bm{y}^{n_T}$ for tolerances
$tol_1=10^{-5}$ (left) and $tol_2=10^{-7}$ (right) using WENO3
with mesh size $\Delta x=0.005$.}
\label{pic:ctr2}}
\end{figure}

\begin{figure}[ht!]
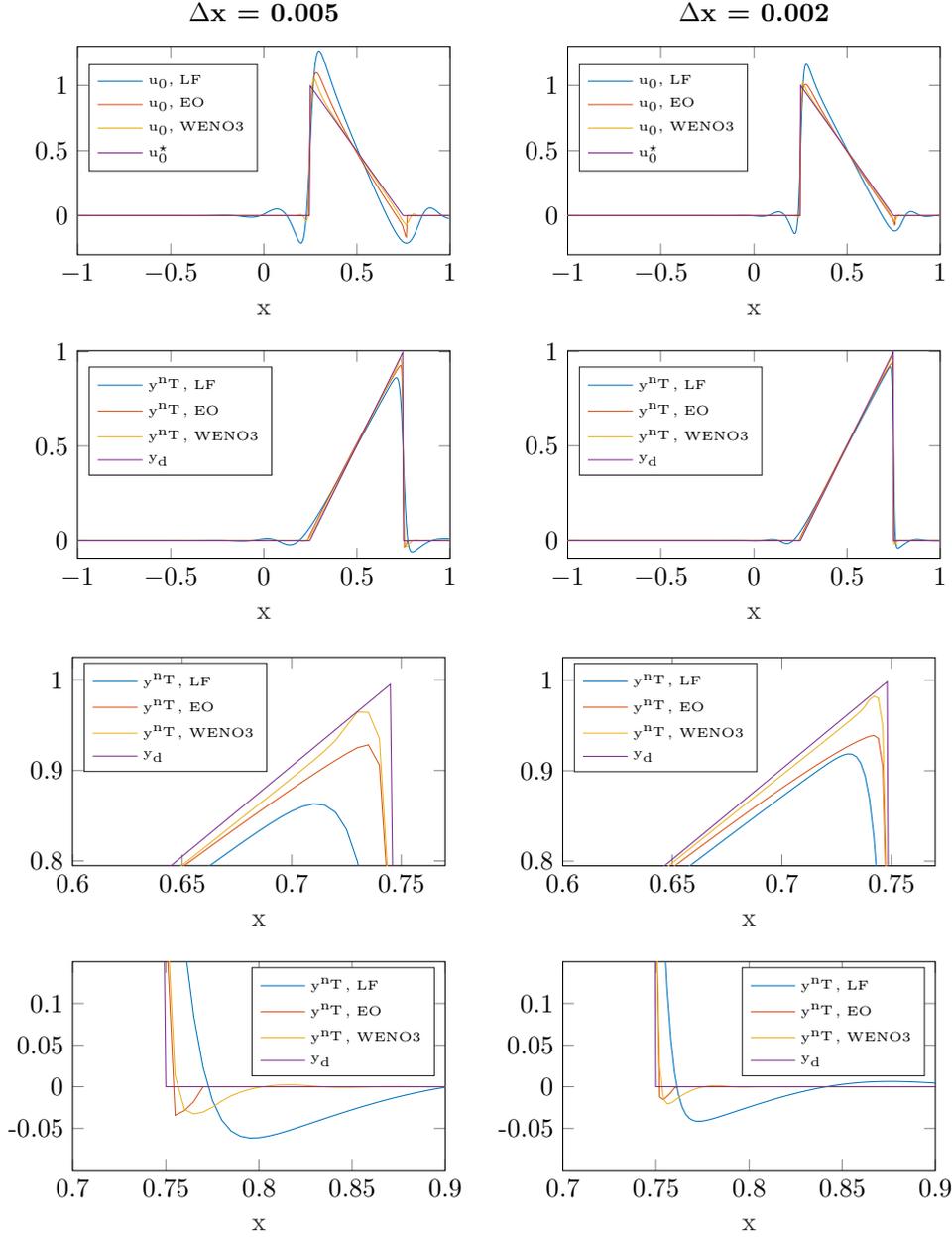

\centering
\definecolor{mycolor1}{rgb}{0.00000,0.44700,0.74100}%
\definecolor{mycolor2}{rgb}{0.85000,0.32500,0.09800}%
\definecolor{mycolor3}{rgb}{0.92900,0.69400,0.12500}%
\definecolor{mycolor4}{rgb}{0.49400,0.18400,0.55600}%
\hspace*{-1cm}\input{pic5.tikz}\\
\vspace{-5mm}
\hspace*{-1cm}\input{pic6.tikz}
\vspace{-5mm}
\parbox{13cm}{
\caption{Optimal control problem. Computed optimal
control functions $\bm{u}_0$ (top) and corresponding
state solution $\bm{y}^{n_T}$ (above the middle) with a zoom into
the shock region (below the middle, bottom) for $50$
iterations of the gradient based optimization algorithm, using
LF, EO, and WENO3 scheme with mesh size $\Delta x=0.005$ (left)
and $\Delta x=0.002$ (right).}
\label{pic:ctr3}}
\end{figure}

\begin{table}[ht!]
\begin{center}
\begin{tabular}{|l|c|c|c|c|}\hline
\rule{0mm}{0.4cm} & LF & EO & WENO3 & WENO3, $\bm{u}_0=0$ \\ \hline
\rule{0mm}{0.4cm} $\Delta x=0.005$   &&&& \\
\hspace{0.3cm} $\text{log} (J_0)$    & $-4.68$ & $-5.76$ & $-7.30$ & $-2.48$ \\
\hspace{0.3cm} $\text{log} (J_{50})$ & $-6.14$ & $-7.80$ & $-8.01$ & $-6.14$ \\ \hline
\rule{0mm}{0.4cm} $\Delta x=0.002$   &&&& \\
\hspace{0.3cm} $\text{log} (J_0)$    & $-5.46$ & $-6.47$ & $-8.34$ & \\
\hspace{0.3cm} $\text{log} (J_{50})$ & $-7.02$ & $-8.55$ & $-8.96$ & \\
\hline
\end{tabular}
\vspace{-2mm}
\parbox{13cm}{
\caption{Optimal control problem. Logarithmic values of the objective functional
(\ref{prob3_obj}) at the beginning and after $50$ iterations of the
optimization algorithm, $J_0$ and $J_{50}$, respectively. For
comparison, values for an initial control $\bm{u}_0=0$ for WENO3 are shown, too.}
\label{tab-prob3-objfunc}}
\end{center}
\end{table}

\section{Summary}
We have developed a novel adjoint WENO3 scheme to provide approximations of the gradient for optimal control problems governed by hyperbolic conservation laws and proved third-order consistency in space for sufficiently smooth solutions. The adjoint WENO3 method is able to sharply resolve discontinuities of reversible solutions. For an exemplary optimal control problem with discontinuous target, the method works very well and outperforms common first-order methods as the Lax-Friedrichs and Engquist-Osher schemes.

\section{Acknowledgement}
This work was supported by the Graduate School CE within the Centre for Computational
Engineering at Technische Universit\"at Darmstadt and by the Deutsche Forschungsgemeinschaft
(DFG, German Research Foundation) within the collaborative research centre
TRR154 {\em ``Mathematical modeling, simulation and optimisation using
the example of gas networks''} (Project-ID 239904186, TRR154/2-2018, TP B01).

\bibliographystyle{plain}
\bibliography{biboptctrweno}
\end{document}